\documentclass[lettersize,journal]{IEEEtran}
\usepackage{amsmath,amsfonts}
\usepackage{algorithmic}
\usepackage{array}
\usepackage[caption=false,font=normalsize,labelfont=sf,textfont=sf]{subfig}
\usepackage{textcomp}
\usepackage{stfloats}
\usepackage{url}
\usepackage{verbatim}
\usepackage{graphicx}
\def\BibTeX{{\rm B\kern-.05em{\sc i\kern-.025em b}\kern-.08em
    T\kern-.1667em\lower.7ex\hbox{E}\kern-.125emX}}
\usepackage{balance}

\usepackage{amssymb,amsmath,amsfonts,amsthm,enumerate,color}
\usepackage[colorlinks=false, pdfborder={0 0 0}]{hyperref}
\usepackage{url}            
\usepackage{booktabs}       
\usepackage{nicefrac}       
\usepackage{microtype}      
\usepackage{cite}
\usepackage{mathtools}
\usepackage{bm}
\usepackage{amsmath}
\usepackage{amssymb}
\usepackage{amsthm}
\usepackage{algorithm,algorithmic}
\usepackage{subfig}
\usepackage{graphicx}

\usepackage{enumerate}
\newcommand{\mathletter}[1]{
	\expandafter\newcommand\csname b#1\endcsname{\mathbb #1}
	\expandafter\newcommand\csname c#1\endcsname{\mathcal #1}
	\expandafter\newcommand\csname f#1\endcsname{\mathfrak #1}
	\expandafter\newcommand\csname til#1\endcsname{\widetilde #1}
	\expandafter\newcommand\csname ha#1\endcsname{\widehat #1}
	\expandafter\newcommand\csname bf#1\endcsname{\bf #1}
	\expandafter\newcommand\csname s#1\endcsname{\mathsf #1}
}
\def\mathletters#1{\mathlettersB #1,,}
\def\mathlettersB#1,{\ifx,#1,\else\mathletter #1\expandafter\mathlettersB\fi}
\mathletters{A,B,C,D,E,F,G,H,I,J,K,L,M,N,O,P,Q,R,S,T,U,V,W,X,Y,Z}

\newcommand{\mathletterl}[1]{
	\expandafter\providecommand\csname v#1\endcsname{\vec{#1}}
}
\def\mathlettersl#1{\mathlettersC #1,,}
\def\mathlettersC#1,{\ifx,#1,\else\mathletterl #1\expandafter\mathlettersC\fi}
\mathlettersl{a,b,c,d,e,f,g,h,i,j,k,l,m,n,o,p,q,r,s,t,u,v,w,x,y,z}

\def\bea{\begin{equation}\begin{aligned} }
\def\ena{\end{aligned}\end{equation} }
\def\bee{\begin{equation}}
\def\ene{\end{equation}}

\renewcommand{\vec}[1]{\mathbf{#1}}

\def\T{\mathsf{T}}

\newtheorem{assum}{Assumption}
\newtheorem{theorem}{Theorem}
\newtheorem{lemma}{Lemma}

\newcommand{\edit}[1]{{\color{black}#1}}

\makeatletter
\def\@fnsymbol#1{\ensuremath{\ifcase#1\or  \natural \or \dagger\or * \or \ddagger\or
   \mathsection\or \mathparagraph\or \|\or **\or \dagger\dagger
   \or \ddagger\ddagger \else\@ctrerr\fi}}
\makeatother

\usepackage{balance}
\begin{document}
\title{Asynchronous Parallel Policy Gradient Methods for the Linear Quadratic Regulator\thanks{This work was  supported by National Key R\&D Program of China (2022ZD0116700) and National Natural Science Foundation of China (62033006, 62325305) (Corresponding author: Keyou You).}
		\thanks{The authors are with the Department of Automation and BNRist, Tsinghua University, Beijing 100084, China. E-mail: shaxy18@mails.tsinghua.edu.cn, zfr18@mails.tsinghua.edu.cn, youky@tsinghua.edu.cn.}
		\author{Xingyu Sha, Feiran Zhao, Keyou You, \IEEEmembership{Senior Member,~IEEE}}
}



\maketitle

\begin{abstract}
Learning policies in an asynchronous parallel way is essential to the numerous successes of RL for solving large-scale problems. However, their convergence performance is still not rigorously evaluated.  To this end, we adopt the asynchronous parallel zero-order policy gradient  (AZOPG) method to solve the continuous-time linear quadratic regulation problem. Specifically,  multiple workers independently perform system rollouts to estimate zero-order PGs  which are then aggregated in a master for policy updates. As in the celebrated A3C algorithm, each worker is allowed to interact with the master {\em asynchronously}. By quantifying the convergence rate of the AZOPG, we  show its linear speedup property, both in theory and simulation, which reveals the advantages of using asynchronous parallel workers in learning policies. 
\end{abstract}

\begin{IEEEkeywords}
Linear system, linear quadratic regulator, policy gradient, asynchronous parallel methods
\end{IEEEkeywords}

\section{Introduction}

Reinforcement learning (RL) has been regarded as an important approach for solving control problems in many domains, e.g., robotics \cite{nguyen2019review}, drones \cite{azar2021drone}, power grid system \cite{zhang2020deep} and autonomous driving \cite{kiran2022deep}. Among its numerous  successes, policy gradient (PG) methods \cite{sutton1999policy}, which directly parameterize control policy and iteratively search an optimal one, are essential. In sharp contrast to the model-based methods, the PG methods can be directly implemented only based on system rollouts. Moreover, their convergence performance has been quantified by solving classical control tasks, such as the linear quadratic regulator (LQR)  and its variants performance\cite{fazel2018global,mohammadi2022convergence,mohammadi2020linear,malik2020derivative, li2022distributed, zhang2019policy, zhao2022sample,zhao2023global,hu2023toward}, where both the global convergence rate and the sample complexities are established.  

However, those end-to-end PG methods usually require a massive number of system rollouts to learn a reliable policy, which is computational demanding, and thus {\em parallel} PG methods with multiple workers have been proposed to accelerate learning. For example, the celebrated A3C algorithm adopts an asynchronous parallel framework \cite{mnih2016asynchronous}. Though it has achieved extensively empirical successes, an explicit speedup rate remains unclear \cite{shen2023towards}. 
To this end, we evaluate the asynchronous parallel zero-order PG (AZOPG) method by solving the continuous-time LQR problem as well. As in the A3C algorithm, multiple parallel workers simultaneously perform system rollouts in an asynchronous way to estimate PGs which are then sent to a master for policy updates. Under this setting, we  show the {\em linear speedup} rate of the  AZOPG with respect to the  number of workers, both in theory and simulation, which rigorously confirms the advantages of using multiple workers to accelerate PG methods. In particular, the number of system rollouts for each worker is proportional to $\log(1/\epsilon)/M$ to achieve an $\epsilon$-accuracy where $M$ denotes the number of workers. 



Though \cite{fazel2018global, mohammadi2022convergence, mohammadi2020linear} have established the global convergence results for zero-order PG methods, it is unclear how to extend  to our  setting  where we are required to solve a nonconvex optimization problem over a nonconvex constraint set in an asynchronous parallel way. This implies that the AZOPG has to use stale zero-order PG estimates, and renders its convergence analysis much more challenging and involved than that of \cite{mohammadi2022convergence, mohammadi2020linear,fazel2018global}. In this work, we resort to the concentration theory to quantify the accuracy of PG estimates and asynchronicity effects. 
%


Our work is also related to the federated PG methods of \cite{ren2020lqr, wang2023model} for solving the LQR problem. However, their methods require synchronizations among workers per policy update, meaning that the faster workers have to wait for the slow ones and cannot fully exploit the advantages of the parallel setting. Usually, asynchronicity is one of the key challenges to be attacked in the parallel algorithm \cite{zhang2019asyspa}. In fact, the AZOPG is closely related to the asynchronous parallel stochastic optimization methods \cite{bertsekas1989parallel, agarwal2011distributed, recht2011hogwild,lian2015asynchronous}. Unfortunately, none of them can be directly adopted due to the nonconvexity of the PG method for the LQR problem.

The rest of the paper is structured as follows. In Section \ref{sec:2}, we introduce the parameterized LQR problem and describe our asynchronous parallel setting. In Section \ref{sec:3}, we propose the AZOPG method and establish its speedup convergence. The proof of convergence is given in Section \ref{sec:4}. In Section \ref{sec:5}, we numerically illustrate the linear speedup property of the AZOPG, and draw some concluding marks in Section \ref{sec:6}. Technical details can be found in appendices.

\textit{Notations:} Denote $\|\cdot\|_\sF$ as the matrix Frobenius norm. Denote ${\rm Tr}(A)$ as the trace of matrix $A$, and $\langle A,B\rangle={\rm Tr}(A^\T B)$ as the matrix inner product. Let ${\rm vec}(A)\in \bR^{mn}$ denote the vectorized form of matrix $A\in\bR^{m\times n}$. Denote $\sigma_{\min}(A)$ as the smallest singular value of a positive-definite symmetric matrix $A$. Let $\cS^{n-1}\subset\bR^{n-1}$ be the unit sphere of dimension $n-1$. Let $\lfloor x\rfloor$ denote the largest integer no larger than $x$. 

\section{Problem Formulation}\label{sec:2}

This section introduces the parameterized continuous-time LQR problem and some PG methods with the focus on our asynchronous parallel setting.

\subsection{The parameterized continuous-time LQR}
We consider the following infinite-horizon LQR for a continuous-time linear system \cite{anderson2007optimal}
\bea\label{equ:lqr-form}
&\mbox{minimize}_{u}& &\bE_{\zeta\sim\cD} \left[\int_{0}^{\infty} \left(x(t)^{\T}Qx(t) + u(t)^{\T}Ru(t) \right)dt \right]\\
&\mbox{subject to}& &\dot{x}(t) = Ax(t) + Bu(t), x(0)=\zeta,
\ena
where $x(t)\in \bR^n$ is the state vector starting with an initial state vector $x(0)=\zeta$ that is sampled from the distribution $\cD$, $u(t)\in \bR^m$ is the control input vector, and $Q\in \bR^{n\times n}$ and $R\in \bR^{m\times m}$ are two positive definite matrices. 

In the sequel, we always assume that $(A,B)$ is controllable. Then, the optimal policy for \eqref{equ:lqr-form} is in the form of linear state feedback \cite{anderson2007optimal}, i.e., $u(t)=-\sK^* x(t), \sK^*\in \bR^{m\times n}$. This implies that there is no loss of optimality to focus on its policy $K$ parameterized form 
\bea\label{equ:policy-optimization-form}
\mbox{minimize}_{\sK\in\cK_{\rm st}}\ f(\sK),\\
\ena
where $f(\sK)={\rm Tr}(P(\sK)X_0)=\bE[f_\zeta(\sK)]$, $f_\zeta(\sK)={\rm Tr}(P(\sK)\zeta\zeta^\T)$ and $\cK_{\rm st}$ denotes the set of stabilizing policies, i.e., $\cK_{\rm st}=\{\sK\in \bR^{m\times n}| A-B\sK \mbox{ is Hurwitz}\}$. Moreover, $X_0 = \bE \left[\zeta\zeta^\T\right]$ and $P(\sK)$ is the unique solution to the Lyapunov equation 
\edit{
$$(A-B\sK)^\T P(\sK) + P(\sK)(A-B\sK) + \sK^\T R\sK + Q=0.$$}
\subsection{PG methods and our asynchronous parallel setting}

\begin{figure}[!t]
	\centering
	\includegraphics[width=0.8\linewidth]{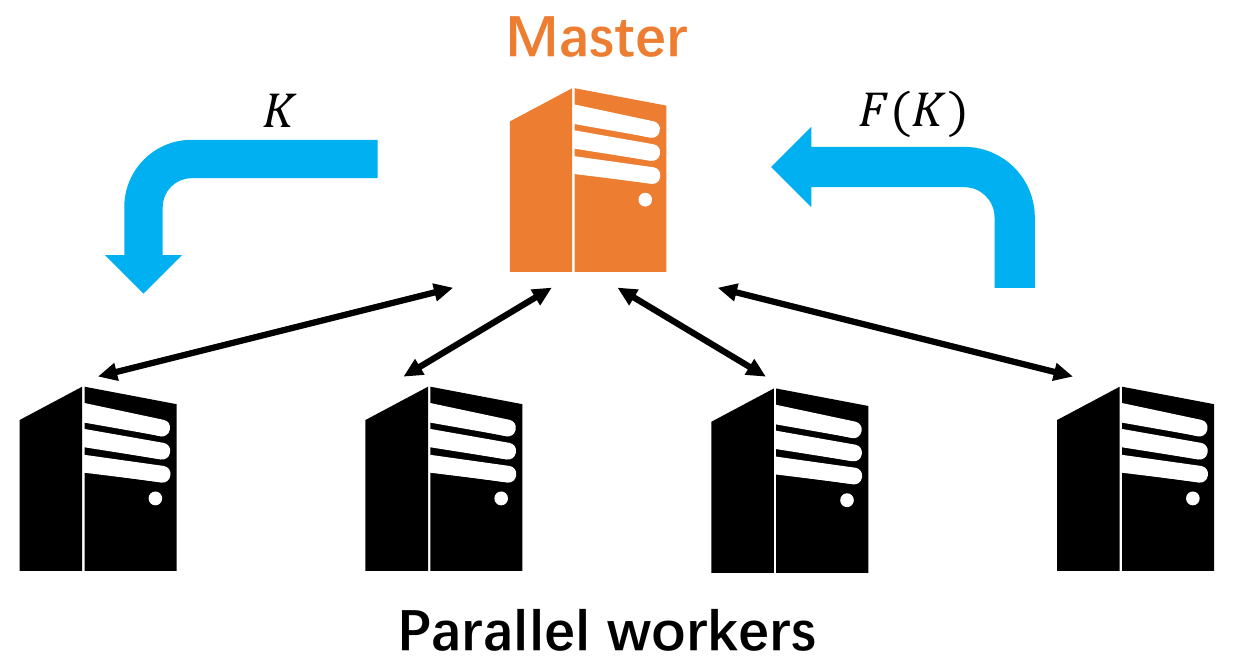}
	\caption{The parallel PG method:  every worker pulls the latest policy $K$ from the master and then computes a PG, which is then pushed back  to the master.}
	\label{star}
\end{figure}

Following \cite{mohammadi2022convergence}, the policy gradient (PG) method aims to  solve \eqref{equ:policy-optimization-form} in an iterative form
\bea\label{modelbasedpg}
\sK_{j+1}=\sK_j - \eta \nabla f(\sK_j), \sK_0\in\cK_{\rm st}, 
\ena
where $\eta$ is a positive stepsize, the PG is given by
\bea\label{equ:policy-gradient}
\nabla f(\sK) = E(\sK)X(\sK),
\ena
and  $E(\sK)=2(R\sK -B^\T P(\sK))$, $X(\sK)$ is the unique solution to \edit{$(A-B\sK)X(\sK) + X(\sK)(A-B\sK)^\T + X_0=0$}. 

Via the concept of {\em gradient dominance}, Ref. \cite{mohammadi2022convergence} shows the global convergence of \eqref{modelbasedpg} to the optimal policy $\sK^*$.  However, the computation of PG in \eqref{equ:policy-gradient} explicitly relies on the system  matrices $(A, B)$. 
%
%
%
If they are not available, Refs. \cite{mohammadi2022convergence, mohammadi2020linear, fazel2018global, malik2020derivative} adopt the zero-order method to estimate PG using the cost observations from system rollouts, including the two-point zero-order method 
\bea
G(\sK)=\frac{f_\zeta (\sK+r\sU)-f_\zeta(\sK-r\sU)}{2r}\sU, \notag
\ena
where $r>0$ is the smooth radius, $\zeta$ is randomly sampled from the distribution $\cD$, $\mbox{vec}(\sU)$ is uniformly sampled from the sphere $\sqrt{mn}\cdot \cS^{mn-1}$, or  its batch version
\bea \label{batchpg}
G_N(\sK)=\frac{1}{N}\sum_{i=1}^N\frac{f_{\zeta_i} (\sK+r\sU_i)-f_{\zeta_i}(\sK-r\sU_i)}{2r}\sU_i,
\ena
where $N$ denotes the batch size, $\{\zeta_i\}_{i=1}^N$ are independently sampled from $\cD$ and $\{\mbox{vec}(\sU_i)\}_{i=1}^N$ are independently sampled. 

\begin{figure}[!t]
	\centering
	\includegraphics[width=\linewidth]{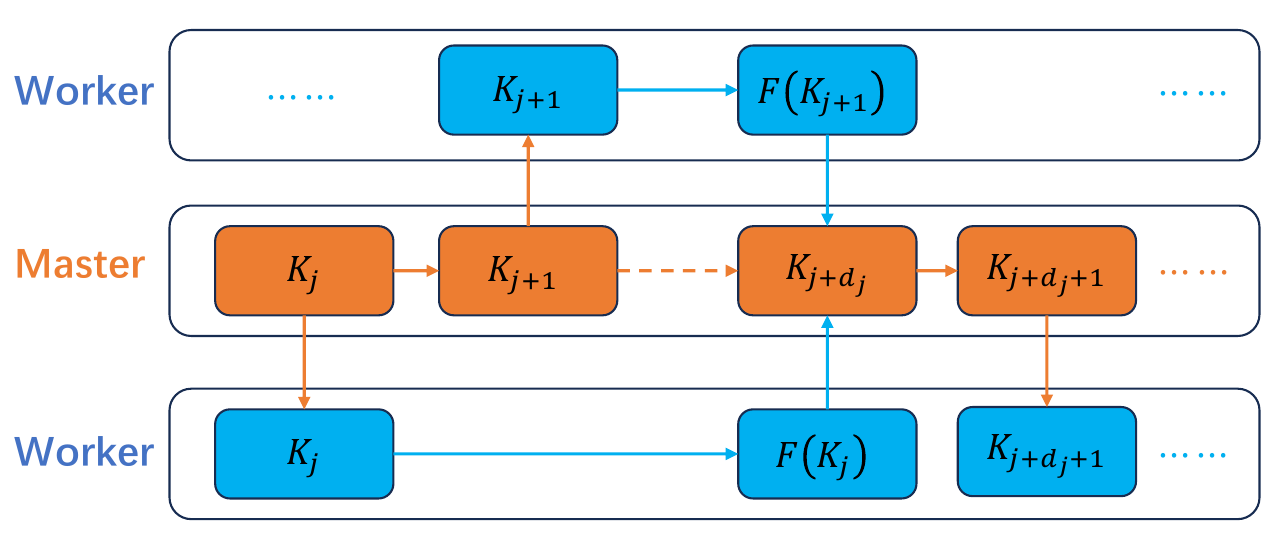}
	\caption{The AZOPG: the master uses stales PG estimates to update policies.}
	\label{asyn-fig}
\end{figure}

In comparison, the batch version reduces the variance of the zero-order PG estimate with a rate $1/N$ and might lead to better convergence performance. 
However, the computational cost of obtaining \eqref{batchpg} is essential as the evaluation of each $f_\zeta(\cdot)$ requires performing a system rollout with a sufficiently large time-horizon $\tau$, i.e., 
\bea\label{truncation-cost}
f_\zeta(\sK) &= \int_{0}^{\infty} \left(x(t)^{\T}Qx(t) + u(t)^{\T}Ru(t) \right)dt \\
& \approx \int_{0}^{\tau} \left(x(t)^{\T}Qx(t) + u(t)^{\T}Ru(t) \right)dt \\
&:=f_\zeta^\tau(\sK)
\ena
where $\dot{x}(t) = Ax(t) + Bu(t), u(t)=-\sK x(t)$ and  $x(0)=\zeta$.  
%
%
%
%

This may be time-consuming in a single computing worker if a relatively large batch size is needed.  
To resolve it, we adopt the master-worker framework of  the well-known A3C \cite{mnih2016asynchronous} in Fig. \ref{star} where multiple workers are employed to simultaneously perform system rollouts to evaluate \eqref{truncation-cost} in a parallel way, hoping to accelerate the computation of  \eqref{batchpg}, and the master aggregates PG estimates from workers to update the policy $\sK$.  It should be noted that such a framework has been widely adopted to accelerate the policy learning process of RL \cite{espeholt2018impala, espeholt2019seed,mnih2016asynchronous} and is not new even for the parameterized LQR problem, see e.g. \cite{wang2023model}.  In sharp contrast to the synchronous updates in \cite{wang2023model}, our iterative method can be {\em asynchronous} in the sense that each worker is not required to wait for others and the master can use stale PGs for policy updates. See Fig. \ref{asyn-fig} for an illustration. That is, such an asynchronous parallel way strictly follows that of the celebrated A3C algorithm \cite{mnih2016asynchronous}. Despite extensive empirical successes of A3C, how to rigorously quantify its convergence advantages largely remains open. The main contribution of this work establishes the linear speedup convergence of the PG method for the LQR problem under the asynchronous parallel setting.



\section{The Asynchronous Parallel Zero-order Policy Gradient Method and Linear Speedup}\label{sec:3}

In this section, we proposes a novel asynchronous parallel zero-order policy gradient (AZOPG) method and establish its linear speedup convergence rate.
\subsection{The AZOPG}
\begin{algorithm}[t!]
\caption{The AZOPG from the view of each worker}
\label{algo-azopg}
\begin{algorithmic}[1]
\REQUIRE distribution $\cD$, smooth radius $r$, time-horizon $\tau$
\STATE Pull the latest policy $\sK$ from the master.
\STATE Uniformly sample a vector ${\rm vec}(\sU)$ from the sphere $\sqrt{mn}\cdot \cS^{mn-1}$ and sample an initial state vector $\zeta$ from $\cD$.
\STATE Set $\sK_1=\sK+r\sU$ and $\sK_2=\sK-r\sU$.
\STATE Compute $f_\zeta^\tau(\sK_k)$ in \eqref{truncation-cost} by  performing two system rollouts with perturbed policies $\sK_k, k\in\{1,2\}$, respectively.
\STATE Estimate PG via the two-point method $F_\zeta^\tau(\sK, \sU)=\frac{1}{2r}(f_\zeta^\tau(\sK_1) - f_\zeta^\tau(\sK_2))\sU$ and then push it to the master.
\end{algorithmic}
\end{algorithm}

Each worker of the AZOPG in Algorithm \ref{algo-azopg} locally performs system rollouts for the master to compute PG estimates of batch version in \eqref{batchpg}. Specifically, it first pulls the latest policy $\sK$ from the master, and randomly generates an initial state vector $\zeta$ and a perturbation matrix $\sU$. Then, it performs two system rollouts to collect the cost of \eqref{truncation-cost} using two perturbed policies $\sK\pm r\sU$ to form a two-point PG estimate $F_\zeta^\tau(\sK, \sU)$ which is pushed back to the master. Note that each worker can repeatedly implement Algorithm \ref{algo-azopg} without waiting for others, and the communication delays between a worker and the master is bounded. Meanwhile, the master keeps receiving two-point PG estimates from workers and simply takes an average once it collects $N$ PG estimates. 

Since it takes time for workers to implement Algorithm \ref{algo-azopg}, the master may use stale two-point PG estimates to update the policy. 
This leads to that the PG method is \edit{mathmatically} given as
\bea\label{master-update}
\overline{G}_j &= \frac{1}{N}\sum_{i=1}^N F_{\zeta_{j,i}}^{\tau}(\sK_{j-d_{j,i}},\sU_{j,i}),\\
\sK_{j+1} &= \sK_j - \eta \cdot \overline{G}_j,
\ena
where $\eta$ is a constant stepsize, $(j,i)$ denotes the index of running system rollouts among all  workers, and $d_{j,i}$ denotes the length of delays in the $j$-iteration. Note that in computing \eqref{master-update}, the master may use multiple two-point PG estimates from the same worker, which is inevitable for our asynchronous setting, and the PG update in \eqref{master-update} usually takes much less time than that of computing a two-point PG estimate.

\subsection{Linear speedup of the AZOPG}
In this subsection, we establish convergence results of the AZOPG method under some reasonable assumptions.

\begin{assum}\label{assum1}
	\textbf{(Initial state distribution): }The distribution $\cD$ has i.i.d. zero-mean entries and unit covariance with bounded support, i.e., $\Vert \zeta \Vert\le \delta$ for some constant $\delta>0$.
\end{assum}

The bounded norm is made only to simplify presentation and can be extended to unbounded case, e.g., Gaussian distributions. We shall validate it in Section \ref{sec:6}.

\begin{assum}\label{assum2}
\textbf{(Bounded computation and communication time)}: For each worker, the  duration between two consecutive time instants of running Algorithm \ref{algo-azopg} is within some interval $[\underline{t},\bar{t}]$ where $0<\underline{t}<\bar{t}<\infty$.
\end{assum}

Assumption \ref{assum2} is mild, since the computation and communication finish in finite time and consume time in practice.  Our result is also based on the smoothness and gradient dominance properties of $f(\sK)$ over its sublevel set $\cQ(a)=\{\sK: f(\sK)\le a\}$ \cite{fatkhullin2021optimizing}, as formalized in the following lemma.
\begin{lemma}\label{lemma-global}  Suppose that $\sK,\sK'\in \cQ(a)$, it holds that
\begin{subequations}
	\begin{alignat}{2}
	f(\sK)-f(\sK^*) &\le \frac{1}{2\mu_1(a)} \|\nabla f(\sK)\|_\sF^2, \label{equ-pl}\\
	\|\nabla f(\sK)-\nabla f(\sK')\|_\sF &\le \mu_2(a) \|\sK-\sK'\|_\sF \label{equ-smooth}
	\end{alignat}
\end{subequations}
where the positive $\mu_1, \mu_2$ also depend on   parameters of the LQR problem in \eqref{equ:lqr-form}.
\end{lemma}

Now, we are in the position to state our main result.

\begin{theorem}\label{theo:main}
Suppose that Assumptions \ref{assum1} and \ref{assum2} hold. Given a desired accuracy $\epsilon>0$ and an initial policy $\sK_0\in \cQ(a)$, let the time-horizon $\tau$, the smooth radius $r$, the batch size $N$, the number of workers $M$ and the stepsize $\eta$ satisfy that
\edit{
\bea\label{parameter-conditions}
r &< \min\{r_0(a), \theta_1(a)\sqrt{\epsilon}\},
&&\ \,\tau\ge \theta_2(a) \log\frac{1}{\epsilon r},\\
N&\ge C_1\beta^4\theta_3(a)\log^6(nm) nm,
&&1\le M \le C_0 N,\\
\eta &\le \frac{1}{128\mu_2(a)(C_g^2(a)+TC_g(a))}.
\ena}
Then, the AZOPG achieves $f(\sK^j)-f(\sK^*)\le \epsilon$ in at most
\bea
j\le ({8}/{\eta})\log\left((f(\sK^0)-f(\sK^*))/{\epsilon}\right) \notag
\ena
iterations with probability not smaller than 
\bea\label{failure}
1-j(C_2N^{-\beta}+C_3e^{-N}). \notag
\ena 
Here, $\beta, C_0, C_1, C_2, C_3$ are positive constants, $\mu_1(a)$ and $\mu_2(a)$ are the gradient dominance and smoothness parameters of the function $f$ over the sublevel set $\cQ(a)$, $C_g, r_0, \theta_1, \theta_2, \theta_3$ are functions that depend on the parameters of the LQR problem, $T=C_0\lfloor \bar{t}/\underline{t} \rfloor+1$.
\end{theorem}

To reach an $\epsilon$-accruate policy, the overall number of system rollouts  is $\cO(\log(1/\epsilon))$, which is independent of $M$, and matches the sample complexity of the centralized zero-order PG in \cite{mohammadi2022convergence}. Thus, the average rollout complexity of each worker in the AZOPG is $\cO(\log(1/\epsilon)/M)$, which verifies the linear speedup property \cite{lian2016comprehensive}.

For the centralized model-free LQR problem, authors of \cite{mohammadi2022convergence} have established $N=\widetilde{\cO}(n)$  for their batch zero-order policy updates. However, their theoretical failure probability is not  negligible under certain conditions (see \cite[Remark 5]{mohammadi2022convergence}). To address it, we choose a relatively conservative batch size $N=\widetilde{\cO}(nm)$.

Though the global convergence of zero-order PG methods has been established in \cite{mohammadi2022convergence, fazel2018global,mohammadi2020linear, malik2020derivative}, there are new challenges under asynchronous parallel updates as our AZOPG has to use stale PG estimates. To resolve it, we quantify the accuracy of the asynchronous batch PG  estimates by concentration analysis, and show that the policies $\{\sK_v\}_{v=1}^j$ remain stabilizing during the asynchronous learning process. 

\section{Proof of Theorem \ref{theo:main}}\label{sec:4}
We prove Theorem 1 in this section. First, we provide some preliminary lemmas. Second, we bound the ratio between the norms of two consecutive PGs used in former updates. Then, we split the inner-product of $\langle\overline{G}_j, \nabla f(\sK_j)\rangle $ into four parts and prove that $-\overline{G}_j$ is a descent direction of $f(K_j)$. Finally, we characterize the effect of one-step asynchronous policy update and complete the proof by induction.


\subsection{Preliminary Lemmas}
First, we use Lemma \ref{lemma-time} to show that the delays $d_{i,j}$ in \eqref{master-update} are upper bounded, and the bound depends on the communication and computation time in Assumption \ref{assum2}.
\begin{lemma}\label{lemma-time}
Under Assumption \ref{assum2}, let $T=C_0\lfloor{\bar{t}}/{\underline{t}} \rfloor+1$. If $M\le C_0N$, then we obtain that $d_{i,j}\le T$ for all $i,j$.
\end{lemma}
\begin{proof}
The time interval between the pull and push operations of any worker is at most $\bar{t}$. During this time interval,  other workers can at most push $\lfloor \frac{\bar{t}(M-1)}{\underline{t}} \rfloor$ messages to the master. Thus, the master can update at most $\lfloor \frac{\bar{t}(M-1)}{\underline{t}N} \rfloor+1$ iterations within the interval, which yields the bound of $d_{i,j}$ for the case $M\le C_0N$.
\end{proof}

Then, we introduce the PG of $f_\zeta(\sK)$, i.e., $\nabla f_\zeta(\sK)=E(\sK)X_\zeta(\sK)$, where $E(\sK)$ is as defined in \eqref{equ:policy-gradient}, and $X_\zeta(\sK)$ is the solution to $(A-B\sK)X + X(A-B\sK)^\T + \zeta\zeta^\T=0$. As mentioned before, the key of our analysis is to prove that $-\overline{G}_j$ is a descent direction of $f(\sK_j)$. Following \cite{mohammadi2022convergence}, we define the unbiased asynchronous gradient estimate $\widehat{G}_j$ as follows,
\bea
\widehat{G}_j = \frac{1}{N}\sum_{i=1}^N \langle \nabla f_{\zeta_{j,i}}(\sK_{j-d_{j,i}}), \sU_{j,i}\rangle  \sU_{j,i}.\notag
\ena
where the expectation of each component in $\widehat{G}_j$ satisfies that $\bE_{\zeta_{j,i}, \sU_{j,i}} \langle \nabla f_{\zeta_{j,i}}(\sK_{j-d_{j,i}}), \sU_{j,i}\rangle  \sU_{j,i}=\nabla f(\sK_{j-d_{j,i}})$. To study the property of $\widehat{G}_j$, we further decompose $\widehat{G}_j$ as:
\bea\label{define-gradient}
\widehat{G}_j &=\widehat{\nabla}_{j,1} + \widehat{\nabla}_{j,2} + \widehat{\nabla}_{j,3}, \\
\widehat{\nabla}_{j,1} &= \frac{1}{N}\sum_{i=1}^N \langle E_{j-d_{j,i}} (\widehat{X}_{j-d_{j,i}} - X_{j-d_{j,i}}), \sU_{j,i}\rangle  \sU_{j,i},\\
\widehat{\nabla}_{j,2} &= \frac{1}{N}\sum_{i=1}^N \langle \nabla f(\sK_{j-d_{j,i}}) - \nabla f(\sK_j), \sU_{j,i}\rangle  \sU_{j,i},\\
\widehat{\nabla}_{j,3} &= \frac{1}{N}\sum_{i=1}^N \langle \nabla f(\sK_{j}), \sU_{j,i}\rangle  \sU_{j,i},\\
\ena
where $E_{j-d_{j,i}}$ and $X_{j-d_{j,i}}$ are the shorthand notations of $E(\sK_{j-d_{j,i}})$ and $X(\sK_{j-d_{j,i}})$, respectively, and $\widehat{X}_{j-d_{j,i}}=X_{\zeta_{j,i}}(\sK_{j-d_{j,i}})$. In \eqref{define-gradient}, $\widehat{\nabla}_{j,1}$ stands for the error caused by random initial state, $\widehat{\nabla}_{j,2}$ is resulted by asynchronous updates, and $\widehat{\nabla}_{j,3}$ is an unbiased estimate for the gradient $\nabla f(\sK_j)$. 

Next, we show that for a sufficiently large $N$, the error terms caused by $\widehat{\nabla}_{j,1}$ and $\widehat{\nabla}_{j,2}$ are bounded when projected to $\nabla f(\sK_j)$, the inner product $\langle \widehat{\nabla}_{j,3}, \nabla f(\sK_j)\rangle$ concentrates around its expectation $\Vert\nabla f(\sK_j) \Vert_\sF^2$, and the norm of $\widehat{G}_j$ is bounded. The proofs are based on concentration analysis and are deferred to Appendix A, B and C.

\begin{lemma}\label{lemma3}
Let $d_j= \max_{1\le i\le N} d_{j,i}$.  Define $\Omega_j$ as the event that the following inequalities all holds:
\begin{align}
\left|\langle \widehat{\nabla}_{j,1}, \nabla f(\sK_j)\rangle \right| &\le 
\frac{1}{8}  \Vert \nabla f(\sK_j) \Vert_\sF \max_{0\le v\le d_j} \Vert \nabla f(\sK_{j-v})\Vert_\sF, \label{omega-1}\\
\left|\langle \widehat{\nabla}_{j,2}, \nabla f(\sK_j)\rangle \right| &\le 
2\Vert \nabla f(\sK_j) \Vert_\sF \notag \\
&\quad \max_{0\le v\le d_j} \Vert \nabla f(\sK_{j-v})- \nabla f(\sK_{j})\Vert_\sF, \label{omega-2} \\
\langle \widehat{\nabla}_{j,3}, \nabla f(\sK_j)\rangle  &\ge 
\frac12 \Vert \nabla f(\sK_j) \Vert_\sF^2, \label{omega-3} \\
\Vert \widehat{G}_{j} \Vert &\le C_g(a) \max_{0\le v\le d_j} \Vert \nabla f(\sK_{j-v})\Vert_\sF. \label{omega-4} 
\end{align}
where $C_g(a)=C_4 \theta'(a) \beta mn\log N+2\sqrt{mn}$, and $\theta'$ is a positive function depend on LQR parameters. 

Then, for any sequence $\{\sK_v\}_{v=0}^j\subseteq \cQ(a)$ and $N\ge C_1\beta^4\theta_3(a)\log^6(nm) nm$, the event $\Omega_j$ happens with probability at least $1-C_2N^{-\beta}-C_3e^{-N}$.
\end{lemma}

Lemma \ref{lemma4} shows that $\widehat{G}_j$ is \edit{$\sqrt{\epsilon}$-close} to $\overline{G}_j$. The proof can be found in Appendix D.
\begin{lemma}\label{lemma4}
Supposing that $r < \min\{r(a), \theta_3(a)\sqrt{\epsilon}\}$, $\tau \ge \theta_2(a) \log(1/(\epsilon r))$ and $\{\sK_v\}_{v=0}^j\subseteq \cQ(a)$, it holds that
\bea\label{epsilon-close}
\Vert \overline{G}_j - \widehat{G}_j \Vert_\sF \le \sqrt{2\mu_1(a)\epsilon}\min\{1/8, C_g(a)\}.
\ena
where $C_g(a)$ is as defined in Lemma \ref{lemma3}.
\end{lemma}

\subsection{Bounding the ratio of consecutive gradients}
In this subsection, we provide an upper bound of $\|\nabla f(\sK_{j-1})\|_\sF^2/\|\nabla f(\sK_{j})\|_\sF^2$. 
\begin{lemma}
Let $\rho=1+1/T$. Suppose that the event $\bigcap_{0\le v\le j-1} \Omega_v$ has happened, $\epsilon< f(\sK_v)-f(\sK^*)\le a$ holds for all $0\le v\le j$ and the AZOPG parameters $\tau, r, N, M, \eta$ satisfy the conditions in \eqref{parameter-conditions}. 

Then, for all $0\le v\le j$, it holds that
\bea\label{lemma1:result}
\|\nabla f(\sK_{v-1})\|_\sF^2 \le \rho\|\nabla f(\sK_{v})\|_\sF^2.
\ena
\end{lemma}

\begin{proof}
From the fact that $f(\sK_{v})-f(\sK^*)> \epsilon$ and the gradient dominance condition \eqref{equ-pl}, for all $0 \le v\le j$, we obtain that 
\bea\label{lemma2-1}
\sqrt{2\mu_1(a)\epsilon}< \Vert\nabla f(\sK_{v})\Vert_\sF.
\ena
Then, we begin to prove \eqref{lemma1:result} by induction. First, we obtain that 
\bea\label{lemma1:step1}
&\|\nabla f(\sK_{v-1})\|^2_\sF - \|\nabla f(\sK_{v})\|^2_\sF \\
&\overset{(a)}{\le} 2\|\nabla f(\sK_{v-1})\|_\sF \|\nabla f(\sK_{v})-\nabla f(\sK_{v})\|_\sF \\
&\overset{(b)}{\le} 2\mu_2(a)\|\nabla f(\sK_{v-1})\|_\sF \|\sK_{v}-\sK_{v-1}\|_\sF \\
&= 2\mu_2(a)\eta\|\nabla f(\sK_{v-1})\|_\sF \|\overline{G}_{v-1}\|_\sF \\
&\overset{(c)}{\le} 2\mu_2(a)\eta\|\nabla f(\sK_{v-1})\|_\sF \left(C_g \sqrt{2\mu_1(a)\epsilon}+ \|\widehat{G}_{v-1}\|_\sF  \right) \\
&\overset{(d)}{\le} 2\mu_2(a)\eta\|\nabla f(\sK_{v-1})\|_\sF \left(C_g(a)\Vert\nabla f(\sK_{v-1})\Vert_\sF \right.
\\&\quad+ \left. C_g(a) \max_{0\le u\le d_{v-1}} \Vert \nabla f(\sK_{v-u-1})\Vert_\sF  \right),
\ena
where step (a) is from the inequality $\Vert a\Vert^2-\Vert b\Vert^2\le 2\Vert a\Vert\Vert b-a\Vert$, in step (b) we adopt \eqref{equ-smooth}, step (c) follows from \eqref{epsilon-close}, and step (d) from \eqref{omega-4} and \eqref{lemma2-1}.

For the basic case, we have $d_0=0$. Let $v=1$ in \eqref{lemma1:step1} and obtain that
\bea
\|\nabla f(\sK_{0})\|^2_\sF - \|\nabla f(\sK_{1})\|^2_\sF&\le 4\mu_2(a)C_g(a) \eta\|\nabla f(\sK_{0})\|^2_\sF\\
&\le \frac{\rho-1}{\rho} \|\nabla f(\sK_{0})\|^2_\sF, \notag
\ena
where the last inequality holds since 
$\frac{\rho-1}{\rho}\ge \frac{\rho^{1/2}-1}{\rho}\ge\frac{1}{4T}>\frac{1}{64T}$. Thus, $\|\nabla f(\sK_{0})\|^2_\sF\le \rho \|\nabla f(\sK_{1})\|^2_\sF$. 

For the induction step, assume that $\|\nabla f(\sK_{u-1})\|^2_\sF\le \rho \|\nabla f(\sK_{u})\|^2_\sF$ holds up to stage $v-1$, which yields that
\bea\label{lemma1:step2}
\max_{0\le u\le d_{v-1}} \Vert \nabla f(\sK_{v-u-1})\Vert_\sF \le \rho^{T/2} \|\nabla f(\sK_{v-1})\|_\sF.
\ena
Then, we substitute \eqref{lemma1:step2} into \eqref{lemma1:step1}. Since $\frac{\rho-1}{\rho^{T/2+1}}\ge\frac{\rho^{1/2}-1}{\rho^{T/2+1}}=\frac{1}{\rho T(1+1/T)^T}\ge \frac{1}{4\cdot T \cdot e}\ge\frac{1}{12T}>\frac{1}{64T}$, it holds that 
\bea \notag
&\|\nabla f(\sK_{v-1})\|^2_\sF - \|\nabla f(\sK_{v})\|^2_\sF \\
&\le 4\mu_2(a)C_g(a)\rho^{T/2} \eta\|\nabla f(\sK_{v-1})\|^2_\sF \le \frac{\rho-1}{\rho}\|\nabla f(\sK_{v})\|^2_\sF, \\
\ena
which completes the proof.
\end{proof}

\subsection{Descent Update Direction}
In this section, we prove that $-\overline{G}_j$ is a descent direction of $f(\sK_j)$, i.e., $\langle-\overline{G}_j, \nabla f(\sK_j)\rangle <0$.
\begin{lemma}\label{lemma-descent}
Suppose that the event $\bigcap_{0\le v\le j-1} \Omega_v$ has happened, $\epsilon< f(\sK_v)-f(\sK^*)\le a$ holds for all $v\le j$ and the AZOPG parameters $\tau, r, N, M, \eta$ satisfy the conditions in \eqref{parameter-conditions}. Then, it holds that
\begin{align}
-\langle \nabla f(\sK_j), \overline{G}_j\rangle &\le-\frac{1}{16}\|\nabla f(\sK_j)\|_\sF^2< 0, \label{correlation2}\\
\Vert\overline{G}_j\Vert_\sF&\le 4C_g(a) \Vert\nabla f(\sK_j)\Vert_\sF. \label{correlation1}
\end{align}
\end{lemma}

\begin{proof}
For LHS of \eqref{correlation2}, it is clear that
\bea\notag
&\langle \nabla f(\sK_j), \overline{G}_j\rangle =  \langle \nabla f(\sK_j), \widehat{\nabla}_{j,1}\rangle  + \langle \nabla f(\sK_j), \widehat{\nabla}_{j,2}\rangle 
\\&\quad\quad\quad\quad+  \langle \nabla f(\sK_j), \widehat{\nabla}_{j,3}\rangle + \langle \nabla f(\sK_j), \overline{G}_j-\widehat{G}_j\rangle.
\ena
We begin with bounding the first term $-\langle \nabla f(\sK_j), \widehat{\nabla}_{j,1}\rangle$:
\bea\notag
\left|\langle \nabla f(\sK_j), \widehat{\nabla}_{j,1}\rangle\right|
&\overset{(a)}{\le} \frac{1}{8}  \Vert \nabla f(\sK_j) \Vert_\sF \max_{0\le v\le d_j} \Vert \nabla f(\sK_{j-v})\Vert_\sF \\
&\overset{(b)}{\le} \frac{\rho^{T/2}}{8}  \Vert \nabla f(\sK_j) \Vert_\sF^2
\overset{(c)}{\le} \frac{1}{4}  \Vert \nabla f(\sK_j) \Vert_\sF^2,\\
\ena
where step (a) follows from \eqref{omega-1}, step (b) follows from \eqref{lemma1:result} and step (c) from the fact that $\left(1+\frac{1}{T}\right)^{T/2}\le \sqrt{e}<2$.

To bound the second term $\langle \nabla f(\sK_j), \widehat{\nabla}_{j,2}\rangle$, we consider $\max_{0\le v\le d_j} \Vert \nabla f(\sK_{j-v})- \nabla f(\sK_{j})\Vert_\sF$:
\bea\label{lemma3-1}
&\max_{0\le v\le d_j} \Vert \nabla f(\sK_{j-v})- \nabla f(\sK_{j})\Vert_\sF\\
&\le \mu_2(a) \max_{0\le v\le d_j} \Vert \sK_{j-v}- \sK_{j}\Vert_\sF
\le \mu_2(a)\eta \sum_{v=j-T}^{j-1}\left\| \overline{G}_v \right\|_\sF\\
&\overset{(a)}{\le} \mu_2(a)C_g(a)\eta \sum_{v=j-T}^{j-1}\max_{0\le w\le d_v}\left\| \nabla f(\sK_{v-w}) \right\|_\sF
\\ &\quad+ \mu_2(a)C_g(a)T\eta \sqrt{2\mu_1(a)\epsilon}\\
&\overset{(b)}{\le} \mu_2(a)C_g(a)T\rho^{T}\eta \left\| \nabla f(\sK_j) \right\|_\sF
+ \mu_2(a)C_g(a)T\eta\left\| \nabla f(\sK_j) \right\|_\sF\\
&\overset{(c)}{\le} 4\mu_2(a)C_g(a)T\eta \left\| \nabla f(\sK_j) \right\|_\sF, \\
\ena
where step (a) follows from \eqref{epsilon-close}, step (b) follows from \eqref{lemma1:result}  and \eqref{lemma2-1}, and step (c) from the fact that $\left(1+\frac{1}{T}\right)^{T}\le e<3$. Based on \eqref{lemma3-1} and \eqref{omega-3}, we obtain that
\bea\notag
|\langle \nabla f(\sK_j), \widehat{\nabla}_{j,2}\rangle|\le 8T\mu_2(a)C_g(a)\eta \left\| \nabla f(\sK_j) \right\|_\sF^2
\ena

The bound of the third term is in \eqref{omega-4}. Based on \eqref{epsilon-close}, \eqref{equ-pl}, the fourth term is bounded by
\bea \notag
|\langle \overline{G}_j-\widehat{G}_j, \nabla f(\sK_j)\rangle|  \le 
			\frac{1}{8} \Vert \nabla f(\sK_j) \Vert_\sF^2.
\ena
Summarizing the above inequalities yields that
\bea\notag
&-\langle \nabla f(\sK_j), \widehat{G}_j\rangle 
\le -\frac12 \Vert \nabla f(\sK_j) \Vert_\sF^2 + \frac{1}{4}\Vert \nabla f(\sK_j) \Vert_\sF^2 \\
& \quad+ 8\mu_2(a)C_g(a)T\eta\|\nabla f(\sK_j)\|_\sF^2 + \frac{1}{8}\Vert \nabla f(\sK_j) \Vert_\sF^2\\
& \le -\frac{1}{8}\|\nabla f(\sK_j)\|_\sF^2 + 8\mu_2(a)C_g(a)T\eta\|\nabla f(\sK_j)\|_\sF^2\\
&\le -\frac{1}{16}\|\nabla f(\sK_j)\|_\sF^2.
\ena
which proves \eqref{correlation2}. By similar approaches, eq. \eqref{correlation1} follows directly from \eqref{omega-4}-\eqref{lemma1:result}, which completes the proof.
\end{proof}

\subsection{Proving the result of Theorem \ref{theo:main}}
In this section, we prove Theorem \ref{theo:main} by induction. We first provide a lemma, which indicates that supposing the sequence $\{\sK_v\}_{v=0}^j$ is in $\cQ(a)$, if event ${\Omega}_j$ happens, then $\sK_{j+1}$ will remain in $\cQ(a)$.

\begin{lemma}\label{lemma-induction}
	Suppose that the event $\bigcap_{0\le v\le j-1} \Omega_v$ has happened, $\epsilon< f(\sK_v)-f(\sK^*)\le a$ holds for all $v\le j$, and the AZOPG parameters $\tau, r, N, M$ satisfy the conditions in \eqref{parameter-conditions}. If event ${\Omega}_{j}$ happens, it holds that $\sK_{j+1}\in \cQ(a)$.

\end{lemma}
\begin{proof}
Lemma \ref{lemma-descent} indicates that $\overline{G}_j$ is a descent direction of $f(\sK_j)$. By the compactness of $\cQ(a)$ \cite{toivonen1985globally}, there exists $\bar{\eta}_j>0$, for all $\gamma\le\bar{\eta}_j$, $\sK_{\gamma}=\sK_j-\gamma\overline{G}\in\cQ(a)$, and $f(\sK_{\bar{\eta}_j})$ satisfies that $f(\sK_{\bar{\eta}_j})=a$.

We now prove that $\eta<\bar{\eta}_j$. For the sake of contradiction, we suppose that $\eta \ge \bar{\eta}_j$. Since $\sK_{\bar{\eta}_j}\in\cQ(a)$, following from the smoothness of $f(\sK)$ over $\cQ(a)$ \eqref{equ-smooth}, it holds that
\bea\label{lemma4-1}
&f(\sK_{\bar{\eta}_j})\le f(\sK_{j})- \bar{\eta}_j \langle \nabla f(\sK_j), \overline{G}_j\rangle + \frac{\mu_2(a)\bar{\eta}_j^2}{2}\Vert\overline{G}_j\Vert_\sF^2 \\
&\le f(\sK_{j})- \frac{\bar{\eta}_j}{8} \Vert\nabla f(\sK_j)\Vert_\sF^2  
+ 8\bar{\eta}_j\eta\mu_2(a)C_g(a)T\Vert\nabla f(\sK_j)\Vert_\sF^2\\
&\quad+ 8\bar{\eta}_j^2\mu_2(a)C_g^2(a)\Vert\nabla f(\sK_j)\Vert_\sF^2 \\
&\le f(\sK_{j}) - \frac{\bar{\eta}_j}{16} \Vert\nabla f(\sK_j)\Vert_\sF^2 
<f(\sK_{j})\le a.
\ena
Eq. \eqref{lemma4-1} indicates that $f(\sK_{\bar{\eta}_j})< a$, which contradicts $f(\sK_{\bar{\eta}_j})=a$. Thus, we obtain that $\eta<\bar{\eta}_j$ and $\sK_{j+1}=\sK_{\eta}\in\cQ(a)$, which completes the proof.
\end{proof}
We start to proof Theorem \ref{theo:main} by induction with the basic case $\sK_0\in \cQ(a)$. For the inductive step, assume that the event $\bigcap_{0\le v\le j-1} \Omega_v$ has happened. As a consequence, $\sK_v\in \cQ(a)$ holds for all $v\le j$. If then $\Omega_j$ happens, Lemma \ref{lemma-induction} indicates that $\sK_{j+1}\in \cQ(a)$. It follows from the smoothness \eqref{equ-smooth} and the gradient dominance condition \eqref{equ-pl} that
\bea\notag
&f(\sK_{j+1})-f(\sK^*)\\
&\le f(\sK_{j})-f(\sK^*)- \eta \langle \nabla f(\sK_j), \overline{G}_j\rangle + \frac{\mu_2(a)\eta^2}{2}\Vert\overline{G}_j\Vert_\sF^2\\
&\le f(\sK_{j})-f(\sK^*)- \frac{\eta}{16} \Vert \nabla f(\sK_j) \Vert_\sF^2\\
&\le (1-\eta\mu_1(a)/8)(f(\sK_{j})-f(\sK^*)),
\ena
which in conjunction with the probability of $\Omega_j$ in Lemma \ref{lemma3} completes the proof.

\section{Numerical Examples}\label{sec:5}
\begin{figure*}[!t]
	\centering
	\captionsetup{justification=raggedright}
	\subfloat[]{\label{fig1-sec1}\includegraphics[width=0.328\linewidth]{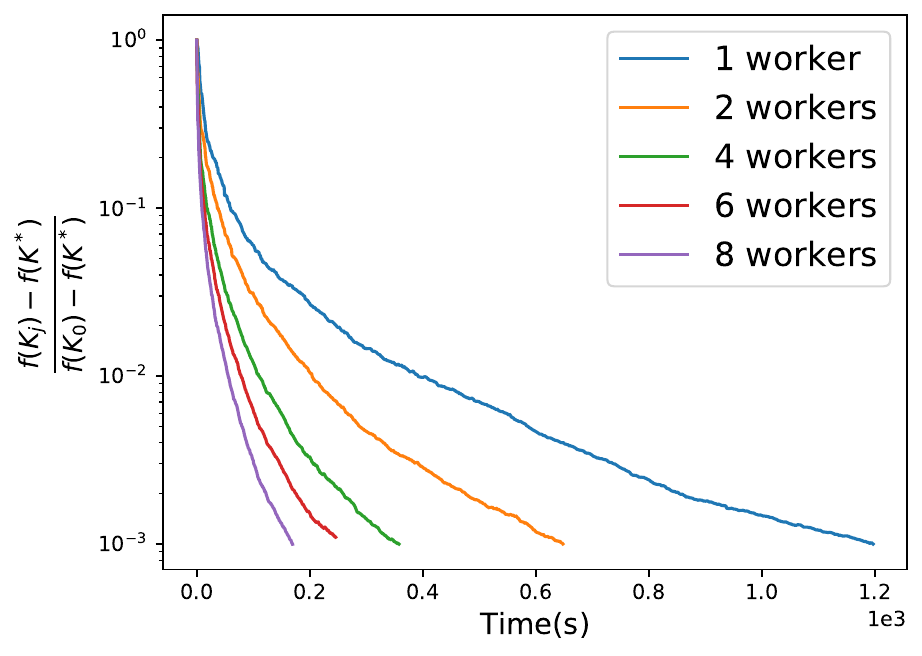}}
	\subfloat[]{\label{fig1-sec2}\includegraphics[width=0.328\linewidth]{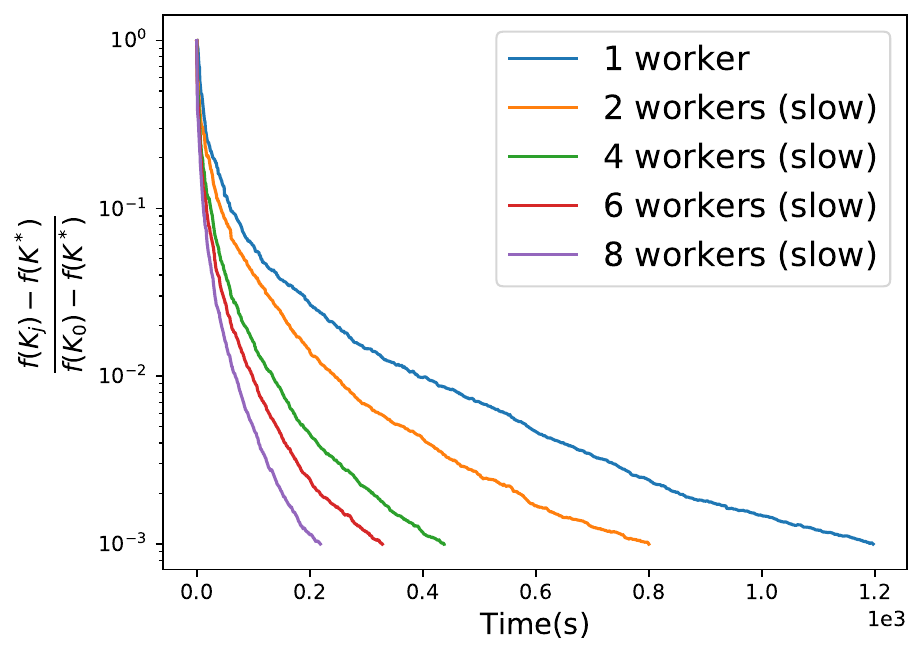}}
	\subfloat[]{\label{fig1-sec3}\includegraphics[width=0.30\linewidth]{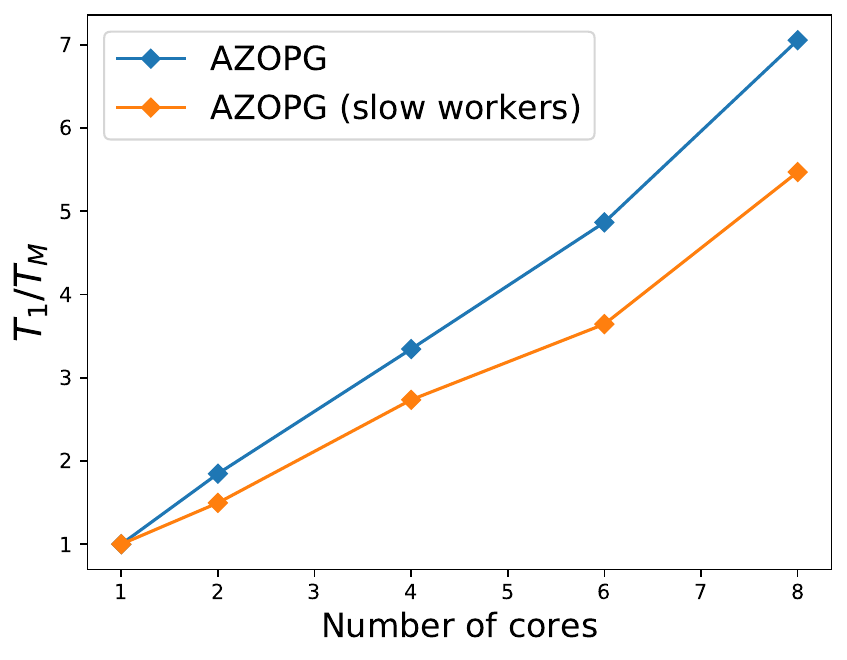}}
	\caption{(a) Convergence performance of AZOPG with different numbers of workers. (b) Convergence performance of AZOPG with half of the workers are slowed down. (c) Speedup in running time with respect to the number of workers for the normal case and the slowed case.
	}
\end{figure*}

We use AZOPG to learn a policy for a mass-spring-damper system of 4 masses, where the state vector $x=[p^\T, v^\T]^\T$ is stacked by the position and velocity vectors, and $A$ and $B$ in \eqref{equ:lqr-form} are given as
\bea\notag
A = \left[\begin{array}{rr}
	0 & I \\
	-T & -T\\
\end{array}\right],~\text{and}~ 
B = \left[\begin{array}{r}
	0 \\ I
\end{array}\right],
\ena
where $0$ and $I$ are zero and identity matrices of $4\times4$ respectively, and $T$ is a Toeplitz matrix with constants $[2, -1, 0, 0]$ on the first row. We set $Q$ and $R$ as identity matrices. The initial vectors $x(0)$ are sampled from normal distribution. The PG algorithm is initialized with $\sK_0=0$.

We implement AZOPG with $1,2,4,6,8$ workers, respectively, where the master and workers are implemented by CPU cores that use \texttt{OpenMPI} in python to communicate. Each worker is assigned to run rollouts for the continuous-time system based on \texttt{control} in python. We set stepsize $\eta=2\times10^{-3}$, smooth radius $r=10^{-5}$ and time-horizon $\tau=100$. Fig. \ref{fig1-sec1} demonstrates the linear convergence of AZOPG with different numbers of workers, implying that the convergence time is roughly inversely proportional to the number of workers. 

In asynchronous parallel implementations, it is common that workers may have different computational speeds. Then, we evaluate the robustness of AZOPG by forcing some workers in the network to slow down. In the 2-, 4-, 6- and 8-worker networks, we manually add a waiting time (100ms) after each local rollout to half of the workers. Fig. \ref{fig1-sec2} demonstrates the convergence of AZOPG with slow workers, which indicates that AZOPG is robust to workers with different computational speeds. It is indicated that AZOPG still keeps an almost linear speedup despite the different update times among the workers.

To illustrate the speedup property, we define the speedup of AZOPG as $T_1/T_M$, where $T_M$ is the running time of AZOPG with $M$ worker(s) to attain $\frac{f(\sK_j)-f(\sK^*)}{f(\sK_0)-f(\sK^*)}\le10^{-3}$. Fig. \ref{fig1-sec3} shows that AZOPG achieves a roughly linear speedup with respect to the number of workers under both aforementioned situations.

\section{Conclusion}\label{sec:6}
In this work, we have introduced an asynchronous parallel zero-order policy gradient (AZOPG) method for the continuous-time LQR problems, which allows multiple parallel workers over a star network to perform local simulations asynchronously. Linear convergence of AZOPG has been established despite the nonconvexity of the problem. We have further validated that the convergence speed of AZOPG increases linearly with respect to the worker number theoretically and practically. Our future works may focus on designing asynchronous actor-critic algorithm for the LQR problem.

\appendix
We first introduce the Orlicz norm and its corresponding properties. The $\psi_\alpha$-norm of a random variable $x$ is given by $\Vert x\Vert_{\psi_\alpha}=\inf_t \{t>0|\bE[\psi_\alpha(|x|/t)]\le 1\}$, where \edit{$\psi_\alpha(x)=e^{x^\alpha}-1$} (linear near $x=0$ when $\alpha\in(0,1)$). The tail bound of the $\psi_\alpha$-norm indicates that
\bea\label{lemma0-6}
\bP\{|x|>t\Vert x\Vert_{\psi_\alpha}\}\le c_\alpha e^{-t^\alpha}.
\ena
We then provide two inequalities that are useful. First, for any $\alpha>0$ and any random variables $x_1, x_2$, it holds that
\bea\label{lemma0-8}
\Vert x_1 x_2 \Vert_{\psi_\alpha}\le C_\alpha\Vert x_1 \Vert_{\psi_{2\alpha}}\Vert x_2 \Vert_{\psi_{2\alpha}}.
\ena
Second, for any $\alpha\in (0,1]$ and any sequence of zero-mean independent variables $x_1, \ldots, x_N$, it holds that
\bea\label{lemma0-7}
\Vert \sum_{i=1}^N x_i \Vert_{\psi_\alpha}\le C_\alpha' \sqrt{N}\log N \max_{1\le i\le N} \Vert x_i \Vert_{\psi_\alpha}.
\ena

\subsection{Proof of \eqref{omega-1} in Lemma \ref{lemma3}}
We first study the $\psi_1$-norm of $\langle E_i(X_i-\bar{X}_i), U_i \rangle$. Let $\langle \widehat{\nabla}_{j,1}, \nabla f(\sK_j)\rangle=1/N\sum_{i=1}^N Y_{j,i}$, $Y_{j,i}=\langle E_{j-d_{j,i}}(\widehat{X}_{j-d_{j,i}}-X_{j-d_{j,i}}), U_i \rangle \langle \nabla f(\sK_j), U_i \rangle$, and $Y_{j,i}'=\langle E_{j-d_{j,i}}(\widehat{X}_{j-d_{j,i}}-X_{j-d_{j,i}}), U_i \rangle$.
From \cite[Lemma 5]{mohammadi2022convergence}, it holds that $\Vert Y_{j,i}' \Vert_{\psi_{1}}
\le \theta'(a)\Vert E_{j-d_{j,i}}^\T \sU_{j,i} \Vert_\sF$
where $\theta'(a)$ depends on the system matrices and initial distribution.
By defining $\theta''(a)=\max_{\sK\in\cQ(a)}\sigma_{\min} (X(\sK))$ \cite[Lemma 16]{mohammadi2022convergence}, we obtain that
\bea\label{lemma1-0}
\Vert Y_{j,i}' \Vert_{\psi_{1}}
&\le \frac{\theta'(a)\theta''(a)}{\sigma_{\min} (X_{j-d_{j,i}})}  \Vert E_{j-d_{j,i}}\Vert_\sF \Vert \sU_{j,i}\Vert_\sF
\\ &\le \theta'(a)\theta''(a)\sqrt{mn} \Vert \nabla f(\sK_{j-d_{j,i}})\Vert_\sF.
\ena
Then, by inequality \eqref{lemma0-8}, we obtain the bound for $\Vert Y_{j,i}\Vert_{\psi_{\frac{1}{2}}}$:
\bea\notag
&\Vert Y_{j,i}\Vert_{\psi_{\frac{1}{2}}}\le C_1\Vert Y_{j,i}' \Vert_{\psi_{1}} \Vert \nabla f(\sK_j) \Vert_{\psi_{1}}\\
&\le C_1\theta'(a)\theta''(a) \sqrt{mn} \Vert\nabla f(\sK_{j-d_{j,i}}) \Vert_\sF \Vert \nabla f(\sK_j) \Vert_\sF
\ena
Applying \eqref{lemma0-7} and $N\ge C_1\beta^4\theta'(a)\theta''(a)\log^6(N)nm$, we obtain that
\bea\label{lemma1-1}
&\|\langle \widehat{\nabla}_{j,1}, \nabla f(\sK_j)\rangle\|_{\psi_{\frac{1}{2}}}\le \frac{C'\log N}{\sqrt{N}}\max_{1\le i \le N} \|Y_{j,i}\|_{\psi_{\frac{1}{2}}}\\
&\le \frac{1}{\beta^2\log^2 N} \Vert \nabla f(\sK_j) \Vert_\sF \max_{1\le i \le N} \Vert\nabla f(\sK_{j-d_{j,i}}) \Vert_\sF.
\ena
Then, combine \eqref{lemma1-1} and inequality \eqref{lemma0-6} with $t= \frac{\Vert \nabla f(\sK_j) \Vert_\sF\max_{i} \Vert\nabla f(\sK_{j-d_{j,i}}) \Vert_\sF}{8\|\langle \widehat{\nabla}_{j,1}, \nabla f(\sK_j)\rangle\|_{\psi_{1/2}}}$ to obtain that with probability at least $1-N^{-\beta}$, the inequality holds:
\bea
|\langle \widehat{\nabla}_{j,1}, \nabla f(\sK_j)\rangle|\le \frac{1}{8}\Vert \nabla f(\sK_j) \Vert_\sF\max_{1\le i \le N} \Vert\nabla f(\sK_{j-d_{j,i}}) \Vert_\sF.
\ena	

\subsection{Proof of \eqref{omega-2} and \eqref{omega-3} in Lemma \ref{lemma3}}\label{appen-b}
We start with a lemma.

\begin{lemma}\label{lemma8}
Let $U_1, \ldots, U_N\in \bR^{m\times n}$ be i.i.d. random matrices with each ${\rm vec}(U_i)$ uniformly distributed on the sphere $\sqrt{mn}\bS^{mn-1}$. Then, for any $W_1, \ldots, W_N, W_1', \ldots, W_N'\in \bR^{m\times n}, $ and $t\in (0,1]$, the following inequality holds with probability at least $1-2e^{-cNt^2}$:
\bea\label{lemma8-1}
&\left|\frac{1}{N} \sum_{i=1}^N \langle W_i, U_i \rangle  \langle W_i', U_i \rangle - \frac{1}{N} \sum_{i=1}^N\langle W_i, W_i' \rangle \right| 
\\ &\le t \max_{1\le i\le N} \Vert W_i\Vert_\sF  \Vert W_i'\Vert_\sF.
\ena
\end{lemma}

\begin{proof}
We first prove that $\langle W_i', U_i \rangle \langle W_i, U_i \rangle - \langle W_i', W_i \rangle$ are zero-mean variables with bounded ${\psi_1}$-norm. The expectation of $\langle W_i', U_i \rangle \langle W_i, U_i \rangle$ satisfies that $\bE\langle W_i', U_i \rangle \langle W_i, U_i \rangle= {\rm vec}^\T(W_i')\left(\bE{\rm vec}(U_i){\rm vec}^\T(U_i)\right){\rm vec}(W_i)=\langle W_i', W_i \rangle$,
where $\bE{\rm vec}(U_i){\rm vec}^\T(U_i)=I$. Meanwhile, the inequality \eqref{lemma0-8} and \cite[Theorem 3.4.6]{vershynin2018high} yields that $\Vert\langle W_i', U_i \rangle \langle W_i, U_i \rangle\Vert_{\psi_1}\le \bar{c}_1\Vert\langle W_i', U_i \rangle\Vert_{\psi_2} \Vert \langle W_i, U_i \rangle\Vert_{\psi_2}\le \bar{c}_2\Vert W_i' \Vert_\sF \Vert W_i \Vert_\sF$. By triangle inequality, we obtain that $\Vert\langle W_i', U_i \rangle \langle W_i, U_i \rangle - \langle W_i', W_i \rangle \Vert_{\psi_1}\le \bar{c}_3\Vert W_i' \Vert_\sF \Vert W_i \Vert_\sF$. 

Since  $\langle W_i, U_i \rangle \langle W_i', U_i \rangle-\langle W_i', W_i \rangle$ are zero-mean independent variables, we apply the Bernstein inequality \cite[Corollary 2.8.3]{vershynin2018high} and obtain the result \eqref{lemma8-1}.
\end{proof}

Then,  we apply Lemma \ref{lemma8} with $t=1$ to obtain that the inequality $\left|\langle \widehat{\nabla}_{j,2}, \nabla f(\sK_j)\rangle \right|\le 2\max_i \Vert \nabla f(\sK_{j}) - \nabla f(\sK_{j-d_{j,i}})\Vert_\sF\Vert \nabla f(\sK_j)\Vert_\sF$
holds with probability at least $1-2e^{-\hat{c}N}$. Meanwhile, Lemma \ref{lemma8} with $t=\frac{1}{2}$ yields that the inequality
$
\langle \widehat{\nabla}_{j,3}, \nabla f(\sK_j)\rangle \ge \frac{1}{2}\Vert \nabla f(\sK_j)\Vert_\sF^2
$
holds with probability at least $1-2e^{-\hat{c}'N/4}$, which completes the proof.

\subsection{Proof of \eqref{omega-4} in Lemma \ref{lemma3}}
Recalling that $\widehat{G}_j=\widehat{\nabla}_{j,1} + \widehat{\nabla}_{j,2} + \widehat{\nabla}_{j,3}$, we 
first discuss $\widehat{\nabla}_{j,2} + \widehat{\nabla}_{j,3}$. Begin by noting that
\bea\label{lemma-last-5}
\sum_{i=1}^N \Vert\langle\nabla f(\sK_{j-d_{j,i}}), \sU_{j,i}\rangle  \sU_{j,i}\Vert_\sF  \le \Vert \bar{U} \Vert_\sF \Vert\bar{s}\Vert,
\ena
where $\bar{U}\in\bR^{mn\times N}$ contains $\mbox{vec}(\sU_{j,i})$ in the $i$th column, $\Vert \bar{U} \Vert_\sF=\sqrt{mnN}$, and $\bar{s}\in\bR^{N} $ is a vector with the $i$th entry $\langle\nabla f(\sK_{j-d_{j,i}}), \sU_{j,i}\rangle$. Then, we apply Lemma \ref{lemma8} with $t=1$ to obtain that
with probability at least $1-2e^{-\tilde{c}N}$, the following inequality holds:
\bea\label{lemma-last-1}
\Vert\widehat{\nabla}_{j,2} + \widehat{\nabla}_{j,3}\Vert_\sF \le 2\sqrt{mn}  \max_{1\le i\le N}\Vert \nabla f(\sK_{{j-d_{j,i}}}) \Vert_\sF.
\ena
Then, we bound $\Vert\widehat{\nabla}_{j,1}\Vert_\sF$. Similar to \eqref{lemma-last-5}, we have
\bea\label{lemma-last-0}
\sum_{i=1}^N \Vert\langle E_{j-d_{j,i}} (\widehat{X}_{j-d_{j,i}} - X_{j-d_{j,i}}), \sU_{j,i}\rangle  \sU_{j,i}\Vert \le \Vert \bar{U} \Vert_\sF \Vert\bar{y}\Vert,
\ena
where $\bar{U}$ has the same definition as in \eqref{lemma-last-5}, and $\bar{y}$ is the vector with the $i$th entry given by $\bar{y}_i=Y_{j,i}'=\langle E_{j-d_{j,i}} (\widehat{X}_{j-d_{j,i}} - X_{j-d_{j,i}}), \sU_{j,i}\rangle$. Then, we obtain that
\bea
&\Vert\Vert \bar{y} \Vert^2\Vert_{\psi_{1/2}}=\Vert\Vert \sum_{i=1}^N \bar{y}_i^2 \Vert\Vert_{\psi_{1/2}}\overset{(a)}{\le} C' N \max_{1\le i\le N}\Vert \bar{y}_i \Vert^2_{\psi_{1}}\\
&\overset{(b)}{\le} C_4\theta''(a) mnN  \max_{1\le i\le N}\Vert \nabla f(\sK_{{j-d_{j,i}}}) \Vert_\sF^2,\\
\ena
where step (a) follows from inequality \eqref{lemma0-7} and \cite[Proposition 2.7.1]{vershynin2018high}, and step (b) from \eqref{lemma1-0}. Using \eqref{lemma-last-0} and applying inequality \eqref{lemma0-6} with $x=\Vert \bar{y} \Vert^2$ and $t=\beta^2\log^2 N$, it holds that
\bea\label{lemma-last-2}
\Vert\widehat{\nabla}_{j,1}\Vert_\sF\le C_4\theta''(a) mn \beta\log N \max_{1\le i\le N}\Vert \nabla f(\sK_{{j-d_{j,i}}}) \Vert_\sF
\ena
with probability at least $1-\tilde{c}' N^{-\beta}$.

Finally, we summarize \eqref{lemma-last-1} and \eqref{lemma-last-2} and obtain that with probability at least $1-2e^{-\tilde{c}N}-\tilde{c}' N^{-\beta}$, the following inequality holds
\bea\notag
\Vert \widehat{G}_j \Vert_\sF \le C_g(a) \max_{0\le v\le d_j} \Vert \nabla f(\sK_{j-v})\Vert_\sF,
\ena
where $C_g(a)=C_4 \theta'(a) \beta mn\log N+2\sqrt{mn}$.
\subsection{Proof of Lemma \ref{lemma4}}
By \eqref{equ-smooth}, it can be proved that for any $\sK\in \cQ(a)$ and $\sU$ with $\Vert \sU \Vert_\sF\le\sqrt{mn}$, we have $\sK+r_0(a)\sU\in \cQ(a)$, where $r_0(a)$ is only dependent on $a$ and the system matrices.

Since $\sK_{j-d_{j,i}}\in\cQ(a)$, if $r\le r_0(a)$, following from \cite[Lemmas 11 and 13]{mohammadi2022convergence}, it holds for $1\le i \le N$ that
\bea\label{lemma0-1}
&\Vert \langle \nabla f_{\zeta_{j,i}}(\sK_{j-d_{j,i}}), \sU_{j,i}\rangle  \sU_{j,i} - F_{\zeta_{j,i}}^\tau(\sK_{j-d_{j,i}},\sU_{j,i})\sU_{j,i} \Vert_\sF\\ 
&\quad\le \frac{\delta^2p_1(2a)e^{-\tau}}{r}+\delta^2r^2p_2(2a),
\ena
where $p_1, p_2$ are positive functions depending on LQR parameters. Then, by defining positive functions 
$
\theta_1(a)=\left(\frac{\sqrt{2\mu_1(a)}\min\{1/8, C_g(a)\}}{2\delta^2p_2(2a)}\right)^{1/2}, 
\theta_2(a)=\log\left(2\frac{\delta^2p_1(2a)}{\sqrt{2\mu_1(a)}\min\{1/8, C_g(a)\}}\right),$
it holds that for any pair of $r<\min\{r_0(a), \theta_1(a)\sqrt{\epsilon}\}$ and $\tau \ge\theta_2(a)\log(1/(r\epsilon))$, the upper bound in \eqref{lemma0-1} becomes no greater than $\sqrt{2\mu_1(a)\epsilon}\min\{1/8, C_g(a)\}$. By triangle inequality, we have 
$\Vert \overline{G}_j - \widehat{G}_j \Vert_\sF\le \frac{1}{N}\sum_{i=1}^N \Vert \langle \nabla f_{\zeta_{j,i}}(\sK_{j-d_{j,i}}), \sU_{j,i}\rangle  \sU_{j,i} - F_{\zeta_{j,i}}^\tau(\sK_{j-d_{j,i}},\sU_{j,i})\sU_{j,i} \Vert_\sF$, which completes the proof.

\bibliographystyle{unsrt}
\bibliography{ref3}

\end{document}